\documentclass[british, a4paper,11pt]{article}

\usepackage[utf8]{inputenc}
\usepackage[british]{babel}

\usepackage{amsthm,amssymb,bbm,stmaryrd}

\usepackage{comment}

\usepackage{libertinus} 
\usepackage[T1]{fontenc}
\usepackage{libertinust1math} 
\usepackage[cal=euler, scr=rsfso]{mathalpha} 
\useosf
\usepackage{cancel}
\usepackage{needspace}

\usepackage{microtype,soul}

\usepackage{numprint}
\usepackage{framed}

\usepackage[explicit]{titlesec}

\usepackage{todonotes}

\titleformat{\section}[block]{\normalfont\large\bfseries\boldmath}{\raggedright\makebox[1em][l]{\thesection.}}{.25em}{#1}
\titleformat{\subsection}[block]{\normalfont\bfseries\boldmath}{\raggedright\makebox[1em][l]{\thesubsection.}}{1em}{#1}
\titleformat{\subsubsection}[runin]{\normalfont\bfseries\boldmath}{\raggedright\makebox[1em][l]{\thesubsubsection.}}{1.5em}{#1.~\hbox{---}}

\renewenvironment{abstract}{%
\begin{center}
\begin{minipage}{.9\textwidth}\linespread{1.05}\selectfont\small
\makebox[5em][l]{\bfseries\abstractname.~\hbox{---}}
\normalfont}
{\par\vspace{1em}
\end{minipage}
\end{center}
}

\usepackage[a4paper,margin=2.5cm]{geometry}
\selectfont

\usepackage{tikz}
	\usetikzlibrary{calc}
	\usetikzlibrary{decorations.markings}
	\usetikzlibrary{arrows,patterns}

\usepackage[font=sf]{caption}
\usepackage{subcaption}

\usepackage{hyperref}
\hypersetup{
	colorlinks=true,
		citecolor=blue!60!black,
		linkcolor=red!60!black,
		urlcolor=green!40!black,
		filecolor=yellow!50!black,
	breaklinks=true,
	pdfpagemode=UseNone,
	bookmarksopen=false,
}

\usepackage{enumitem}

\numberwithin{equation}{section}

\newtheorem{thm}{\bfseries \upshape Theorem}
\newtheorem{lem}{Lemma}
\newtheorem{prop}{Proposition}
\newtheorem{cor}{Corollary}

\theoremstyle{definition}

\newtheorem{defi}{Definition}
\newtheorem*{openproblem}{Open problem}

\newcommand{\R}{\mathbb{R}}

\usepackage{bbm}
\newcommand{\ind}{\mathbbm{1}}

\newcommand{\eqloi}[1][n]{\enskip\mathop{=}^{(d)}_{}\enskip}

\usepackage{mathtools}

\let\originalleft\left
\let\originalright\right
\renewcommand{\left}{\mathopen{}\mathclose\bgroup\originalleft}
\renewcommand{\right}{\aftergroup\egroup\originalright}

\DeclareSymbolFont{extraup}{U}{zavm}{m}{n}
\DeclareMathSymbol{\vardspade}{\mathalpha}{extraup}{81}
\DeclareMathSymbol{\varheart}{\mathalpha}{extraup}{86}
\DeclareMathSymbol{\vardiamond}{\mathalpha}{extraup}{87}
\DeclareMathSymbol{\varclub}{\mathalpha}{extraup}{84}
\DeclareMathOperator{\rmsupp}{\mathrm{supp}}

\makeatletter
\renewcommand*{\@fnsymbol}[1]{\ensuremath{\ifcase#1\or \vardspade \or \varheart\or \vardiamond \or \varclub \or
  \mathsection\or \mathparagraph\or \|\or **\or \dagger\dagger
  \or \ddagger\ddagger \else\@ctrerr\fi}}
\makeatother

\newcommand{\calG}{\mathscr{G}}
\newcommand{\calH}{\mathscr{H}}
\newcommand{\calL}{\mathscr{L}}
\newcommand{\calM}{\mathscr{M}}
\newcommand{\calP}{\mathscr{P}}
\newcommand{\calW}{\mathscr{W}}

\newcommand{\lseg}{\llbracket}
\newcommand{\rseg}{\rrbracket}
\usepackage{booktabs}
\usepackage{longtable}

\usepackage{algorithm}
\usepackage{algpseudocode}

\author{
	Philippe \textsc{Bouafia}\thanks{Fédération de Mathématiques FR3487, CNRS, CentraleSupélec, 3 rue Joliot Curie, 91190 Gif-sur-Yvette, France.\hfill \href{mailto:philippe.bouafia@centralesupelec.fr}{\texttt{philippe.bouafia@centralesupelec.fr}}} 
}

\title{Magnitude and diversity of trees}

\begin{document}

\maketitle

\begin{abstract}
We compute the magnitude (an isometric invariant of metric spaces) of compact $\R$-trees and show that it equals $1 + L/2$, where $L \in [0, \infty]$ denotes the total length. Although length is the only geometric invariant captured by magnitude, we show that diversity-maximizing measures on compact $\R$-trees are more sensitive to the branching structure as they tend to be more concentrated toward the leaves: their support contains no branch points. In the finite case, we further show that maximum diversity on a weighted tree can be computed in polynomial time.
\end{abstract}




{\small
\tableofcontents
}

\section{Introduction}

In this article, we investigate two invariants of metric spaces—\emph{magnitude} and \emph{maximum diversity}—and apply them to several classes of trees. The notion of magnitude has a twofold origin: it arises from theoretical ecology \cite{solow1994measuring} as a way to quantify biodiversity, and its systematic mathematical study was initiated by Leinster, independently motivated by ideas from category theory \cite{Leinster2013Magnitude, LeinsterMeckes2017156193}. Its variant, maximum diversity, has been further studied by, among others, Leinster, Meckes and Roff, in \cite{leinster_meckes_diversity, Meckes2015Magnitude, LeinsterRoff2021}.

To motivate the results, we briefly recall the definition of magnitude for a finite metric space $(X,d)$, where the notion was first introduced. Define the matrix $Z \colon X \times X \to \R$ by $Z(x,y) = e^{-d(x,y)}$; its entries quantify the similarity between pairs of points and are maximal when $x = y$. Generically, $Z$ is invertible, which guarantees the existence of a function $w \colon X \to \R$ solving the linear system $Zw = \ind_X$. The value $w(x)$ may be interpreted as the weight of the point $x$ in $X$, and the total $|X| = \sum_{x \in X} w(x)$ is called the \emph{magnitude} of $X$. This notion can be viewed as a metric analogue of cardinality: if all points are mutually at infinite distance, then $w(x)=1$ for all $x$, and $|X|$ coincides with the cardinality of $X$.

Magnitude is thus a relatively tractable invariant, arising from a linear problem. Due to its category-theoretic origins, it also behaves well under operations such as $\ell^1$-products and wedge sums, making it particularly well suited to the study of trees. The weights $w(x)$ are however not required to be non-negative. By contrast, \emph{maximum diversity} can be regarded as a constrained variant of magnitude in which only non-negative weights are allowed. A probability measure $\mu \colon X \to [0, 1]$ is said to be diversity-maximizing if it minimizes the quantity
\[
\mu^\top Z \mu = \sum_{x,y \in X} e^{-d(x,y)} \mu(x)\mu(y)
\]
which represents the average similarity between two independent samples drawn according to $\mu$. A somewhat counterintuitive feature of this theory is that the support of a diversity-maximizing measure, which one may expect to be highly spread out, need not coincide with the whole space. In the absence of additional geometric information, determining such a measure typically requires enumerating all possible supports $Y \subseteq X$, leading to exponential complexity.

We study these quantities for the following classes of compact metric spaces:
\begin{itemize}
\item \emph{Weighted trees}: graphs in which each edge is assigned a length. Here, the metric space consists of the finite set of vertices.
\item \emph{Simplicial trees}: similar to weighted trees, but with edges included as part of the metric space.
\item \emph{$\R$-trees}: continuous analogues of simplicial trees, which may exhibit highly intricate branching structures. Compact $\mathbb{R}$-trees arise as Gromov--Hausdorff limits of simplicial trees.
\end{itemize}

All these spaces are of negative type and therefore fit naturally into the theory of magnitude and maximum diversity in the compact setting developed by Meckes in~\cite{Meckes2013PositiveDefinite}. Moreover, the latter two classes come equipped with a natural notion of total length. Our first main result shows that this total length is the only geometric quantity detected by magnitude. In the case of $\R$-trees, this length may be infinite, yielding new examples of compact metric spaces of negative type with infinite magnitude. This provides a new answer to a question raised in \cite{Meckes2013PositiveDefinite, Meckes2015Magnitude, LeinsterMeckes2017156193}; a counterexample had previously been constructed in \cite{MeckesLeinster2023ExtremalMagnitude}.

This suggests that much of the finer geometry of trees is captured instead by the \emph{maximum diversity function}, which associates to each $t>0$ the maximum diversity $|(X, td)|_+$ of the rescaled space $(X, td)$. For instance, Meckes showed that, for a general compact metric space, the upper and lower Minkowski dimensions can be recovered from its asymptotic behavior \cite{Meckes2015Magnitude}.

Any compact $\R$-tree can be decomposed into two parts: its skeleton $\operatorname{Sk}(X)$, which carries its length, and its set of leaves $\operatorname{Leaves}(X)$. The skeleton is a relatively simple object of Hausdorff dimension $1$, whereas $\operatorname{Leaves}(X)$ may have Hausdorff dimension greater than $1$. In such cases, the dimension of $X$ is governed by the leaves, and in light of the preceding discussion, one expects diversity-maximizing measures to concentrate on this set.

As noted above, the main difficulty in determining diversity-maximizing measures lies in identifying their support—a problem that remains open even for Euclidean balls. We therefore modestly focus on establishing geometric constraints on this support. Our main result for compact $\R$-trees shows that branch points do not belong to the support of the diversity-maximizing measure. When branch points are dense—as is typical in random trees, such as the Brownian tree—this implies that such measures are supported on a meager set. This phenomenon has a biological interpretation: in an ecosystem, most ancient species, even if they could be revived, would not contribute to maximizing biodiversity.

For weighted trees, by contrast, the theory takes a more computational turn. This is the only setting in which maximum diversity can be computed exactly, and we provide a polynomial-time algorithm to do so.

\paragraph{Plan of the paper.} The structure of the article is as follows. In Section~\ref{sec:max_div_general}, we introduce the notions of magnitude and maximum diversity in the setting of compact metric spaces of negative type. This section largely recalls known results, sometimes with simplified proofs. In contrast to the presentation in the introduction, we begin with maximum diversity, which we find more intuitive. In Section~\ref{sec:magtree}, we introduce the various classes of trees under consideration and compute their magnitude. Section~\ref{sec:max_div_wtree} presents an algorithm for computing the maximum diversity of weighted trees, as well as the associated maximizing measure. The correctness of the algorithm yields results on the support of this measure, in particular showing that branching points do not belong to it when they are at distance at most $\log 2$ from their neighbors. Finally, in Section~\ref{sec:support}, we establish an analogous result in the continuous setting. Most notation, beyond standard conventions, is introduced progressively throughout the paper. Appendix~\ref{sec:notations} provides a summary of the notation used and may be consulted as needed.

\paragraph{Acknowledgments.} The author is grateful to Erick Herbin for helpful discussions.

\section{Maximum diversity and magnitude of compact metric spaces of negative type}
\label{sec:max_div_general}

In this section, we introduce the concepts of \emph{maximum diversity} and \emph{magnitude} of a metric space $X$. Following Meckes in \cite{Meckes2013PositiveDefinite}, the extension of this theory to compact metric spaces becomes more natural upon imposing an additional condition on $X$, namely that it be of \emph{negative type}. This assumption is not overly restrictive: the class of compact metric spaces of negative type includes all compact subspaces of $L^1$ and $L^2$, as well as the various notions of trees that will arise in this work.

\subsection{Maximizing diversity}
\label{subsec:max_div}

Let $(X,d)$ be a compact metric space, where the distance $d$ is allowed to take the value $+\infty$ (although this will not occur in the situations considered in this paper). The \emph{similarity kernel} is the function $Z \colon X \times X \to [0,1]$ defined by $Z(x,y) = e^{-d(x,y)}$.
As the name suggests, $Z(x,y)$ measures the similarity between the points $x$ and $y$; it attains its maximal value precisely when $x = y$.

We denote by $\calM(X)$ the linear space of signed Borel measures on $X$, and by $\calP(X)$ the subset of probability measures. We equip $\calP(X)$ with the weak topology, that is, the coarsest topology for which all maps
\[
\mu \mapsto \int_X f(x)\,\mu(dx), \qquad f \in C(X)
\]
are continuous. Here, $C(X)$ denotes the linear space of continuous functions on $X$, equipped with the supremum norm. Doing so, $\calP(X)$ becomes a compact metrizable space.

To each measure $\mu \in \calM(X)$, we associate the function $Z\mu \colon X \to \R$ defined by
\[
Z\mu(x) = \int_X Z(x,y)\,\mu(dy)
\]
By a standard application of the Lebesgue dominated convergence theorem, the function $Z\mu$ is continuous.
We then define a bilinear form on $\calM(X)$ by
\begin{equation}
\label{eq:def_similarity_form}
\langle \mu, \nu \rangle_{\calW} 
= \int_X Z\mu \, d\nu 
= \int_X \int_X Z(x,y)\,\mu(dx)\nu(dy), 
\qquad \mu, \nu \in \calM(X)
\end{equation}
The symmetry of $Z$ immediately implies that $\langle \cdot, \cdot \rangle_{\calW}$ is symmetric.

When $\mu$ is a probability measure, the quantity $Z\mu(x)$ can be interpreted as the \emph{typicality} of the point $x$: it represents the average similarity between $x$ and a random point of $X$ drawn according to $\mu$. In the same spirit, $\langle \mu, \mu \rangle_{\calW}$ corresponds to the average similarity between two independent random points of $X$, both distributed according to $\mu$. 
Accordingly, diversity is maximized when $\langle \mu, \mu \rangle_{\calW}$ is minimized. We therefore define the \emph{maximum diversity} of $X$ by
\begin{equation}
\label{eq:def_diversity}
|X|_+ = \sup_{\mu \in \calP(X)} \frac{1}{\langle \mu, \mu \rangle_{\calW}}
\end{equation}
This notion of diversity belongs to a broader family of diversity measures introduced by Leinster and Roff in \cite{LeinsterRoff2021} (see also \cite{leinster_meckes_diversity} for the case of finite metric spaces), parametrized by a real number $q \geq 0$. The notion considered here corresponds to the case $q = 2$.
Leinster and Roff proved the existence of probability measures that simultaneously maximize these diversities for all $q \geq 0$. Here, we establish a weaker result for the case $q=2$ alone:  the supremum in~\eqref{eq:def_diversity} is attained by at least one \emph{diversity-maximizing measure}. This follows from the following well-known continuity result.
\begin{lem}
\label{lem:2.1}
Let $X$ be a compact metric space. The map $\mu \in \calP(X) \mapsto \mu \otimes \mu \in \calP(X \times X)$ is continuous.
\end{lem}
\begin{proof}
We need to prove the continuity of all linear maps
\[
\Phi_f \colon \mu \in \calM(X) \mapsto \int_{X \times X} f(x,y)\, \mu \otimes \mu(dx,dy)
\]
for $f \in C(X \times X)$.

When $f$ has separated variables, i.e., $f(x,y)=g(x)h(y)$ for some $g,h \in C(X)$, we may rewrite
\[
\Phi_f(\mu)=\left(\int_X g(x)\, \mu(dx)\right)\left(\int_X h(y)\, \mu(dy)\right)
\]
and continuity of $\Phi_f$ follows immediately from the definition of the weak topology.
The subspace $C(X)\otimes C(X) \subseteq C(X \times X)$ of finite linear combinations of functions with separated variables is dense by the Stone--Weierstrass theorem. Consequently, every $f \in C(X \times X)$ is the uniform limit of a sequence $(f_n)$ in $C(X)\otimes C(X)$.
One then checks that $\Phi_f$ is the uniform limit of the corresponding maps $\Phi_{f_n}$, which implies the continuity of $\Phi_f$.
\end{proof}

\begin{prop}
 A compact metric space has at least one diversity-maximizing measure.
\end{prop}

\begin{proof}
This follows from the continuity of $\mu \mapsto \langle \mu, \mu \rangle_{\calW}$, guaranteed by Lemma~\ref{lem:2.1} and the compactness of $\calP(X)$.
\end{proof}

\subsection{Metric spaces of negative type}
\label{subsec:negtype}

It is necessary to restrict our attention to a suitable class of metric spaces for which the notions of diversity and magnitude (to be introduced in Subsection~\ref{subsec:magnitude_and_diversity}) behave in a sensible manner. In particular, we would like to ensure uniqueness of the diversity-maximizing measure. At the same time, this class should be broad enough to include all relevant examples. A convenient property is that the bilinear form $\langle \cdot, \cdot \rangle_{\calW}$ be positive definite.

A finite metric space $X$ is said to be \emph{positive definite} if its similarity kernel $Z$ defines a positive definite matrix. A compact metric space is called \emph{positive definite} if all of its finite subsets are positive definite; equivalently, if $\langle \cdot, \cdot \rangle_{\calW}$ is positive definite on $\calM_{\mathrm{atom}}(X)$, the subspace of $\calM(X)$ spanned by the Dirac masses $\delta_x$, for $x \in X$. The reference on positive definite metric spaces is \cite{Meckes2013PositiveDefinite}.

A metric space $(X,d)$ is said to be \emph{of negative type} if $(X,\sqrt{d})$ admits an isometric embedding into a Hilbert space. For instance, the interval $[0, L]$ is of negative type, as witnessed by the map $[0, L] \to L^2([0, L])$ that sends $x$ to $\ind_{[0, x]}$. The terminology comes from a characterization due to Schoenberg, which we shall not use here. This property is stronger than positive definiteness, as we now show.

\begin{thm}
\label{thm:posdef}
  Let $X$ be a compact metric space of negative type. The bilinear form $\langle \cdot, \cdot \rangle_\calW$ is positive definite.
\end{thm}

\begin{proof}
  Let $\varphi \colon (X, \sqrt{d}) \to (H, \| \cdot \|)$ be an isometric embedding into some Hilbert space $H$. Let us denote by $\calM_c(H)$ the space of compactly supported signed Borel measures on $H$. We define the following bilinear form on $\calM_c(H)$
  \[
  \calG(\mu, \nu) = \int_H \int_H e^{-\|y - x\|^2} \, \mu(dx) \nu(dy), \qquad \mu, \nu \in \calM_c(H)
  \]
  with a Gaussian kernel. It is related to $\langle \cdot, \cdot \rangle_\calW$ by
  \[
  \langle \mu', \nu' \rangle_{\calW} = \int_X \int_X e^{-\| \varphi(y) - \varphi(x) \|^2} \, \mu'(dx) \nu'(dy) = \calG(\varphi_\# \mu', \varphi_\# \nu')
  \]
  Thus, it suffices to establish the positive definiteness of $\calG$.
  
  {\bf Step 1: $\calG$ is positive semidefinite on $\calM_{\mathrm{atom}}(H)$.} Consider a measure of the form $\mu = \sum_{i=1}^n \alpha_i \delta_{x_i}$. In this case,
\[
\calG(\mu, \mu) = \sum_{i,j=1}^n \alpha_i \alpha_j e^{-\|x_j - x_i\|^2}
\]
and we need to show that this quantity is nonnegative. Since the points $x_1, \dots, x_n$ lie in a finite-dimensional subspace of $H$ (of dimension at most $n$), we may assume without loss of generality that $H = \mathbb{R}^n$ for the remainder of this step.

By the Fourier transform, we have for all $x \in \mathbb{R}^n$,
\[
e^{-\|x\|^2} = \frac{1}{2^n \pi^{n/2}} \int_{\mathbb{R}^n} \exp\!\left(-\frac{\|\xi\|^2}{4}\right) e^{-i \langle \xi, x \rangle}\, d\xi
\]
Substituting this identity into the expression for $\calG(\mu,\mu)$ yields
\[
\calG(\mu, \mu)
= \frac{1}{2^n \pi^{n/2}} \int_{\mathbb{R}^n} \exp\!\left(-\frac{\|\xi\|^2}{4}\right)
\left|\sum_{k=1}^n \alpha_k e^{-i \langle \xi, x_k \rangle}\right|^2 \, d\xi \ge 0
\]
which establishes the claim.

  {\bf Step 2: $\calG$ is positive semidefinite on $\calM_c(H)$.} This follows from a standard density argument. Let $K$ denote the support of a measure $\mu \in \calM_c(X)$, and fix $\varepsilon > 0$. The function $(x,y) \mapsto e^{-\|x-y\|^2}$ is uniformly continuous on $K \times K$, so there exists a modulus of uniform continuity $\delta$ corresponding to~$\varepsilon$.

Let $B_1, \dots, B_n$ be a partition of $K$ into Borel sets of diameter at most $\delta/\sqrt{2}$ (so that each product $B_i \times B_j$ has diameter at most $\delta$). For each $i \in \{1,\dots,n\}$, choose a point $x_i \in B_i$, and define
\[
\nu = \sum_{i=1}^n \mu(B_i)\,\delta_{x_i}
\]
By the first step, we have $\calG(\nu,\nu) \ge 0$.

Moreover,
\[
\calG(\mu, \mu) - \calG(\nu, \nu)
= \sum_{i,j=1}^n \int_{B_i} \int_{B_j}
\left(e^{-\|x-y\|^2} - e^{-\|x_i - x_j\|^2}\right)\,\mu(dx)\mu(dy)
\]
so that, by uniform continuity,
\[
\calG(\mu, \mu) - \calG(\nu, \nu) \ge -\varepsilon\, \|\mu\|(H)^2
\]
We therefore obtain $\calG(\mu, \mu) \ge -\varepsilon \|\mu\|(H)^2$, and since $\varepsilon > 0$ is arbitrary, it follows that $\calG(\mu, \mu) \ge 0$.

{\bf Step 3: $\calG$ is positive definite on $\calM_c(H)$.} Let $\mu \in \calM_c(H)$ be such that $\calG(\mu,\mu)=0$. Since $\calG$ is already known to be positive semidefinite, it follows from the Cauchy--Schwarz inequality that
$\calG(\mu,\nu)=0$ for all $\nu \in \calM_c(H)$.
Therefore,
\[
\int_H \left(\int_H e^{-\|x-y\|^2}\,\mu(dx)\right)\nu(dy)=0
\]
for all $\nu \in \calM_c(H)$, which in particular implies, by choosing $\nu=\delta_y$, that
\[
\int_H e^{-\|x-y\|^2}\,\mu(dx)=0 \qquad \text{for all } y \in H
\]

Let $K$ denote the compact support of $\mu$, and for each $y \in H$ define the function $e_y \colon K \to \mathbb{R}$ by $e_y(x)=e^{-2\langle y,x\rangle}$.
Note that $e_y e_{y'}=e_{y+y'}$ for all $y,y' \in H$ and $e_0 = \ind_K$. It follows that the vector space $E \subseteq C(K)$ spanned by the functions $e_y$ is in fact a real algebra. Moreover, $E$ separates the points of $K$. By the Stone--Weierstrass theorem, $E$ is dense in $C(K)$.

We also observe that for every $f \in E$,
\[
\int_K e^{-\|x\|^2} f(x)\,\mu(dx)=0
\]
Indeed, writing $f=\sum_{i \in I} \lambda_i e_{y_i}$, we obtain
\[
\int_K e^{-\|x\|^2} f(x)\,\mu(dx)
= \sum_{i \in I} \lambda_i \int_K e^{-\|x\|^2} e^{-2\langle y_i,x\rangle}\,\mu(dx)
= \sum_{i \in I} \lambda_i e^{\|y_i\|^2} \int_K e^{-\|x-y_i\|^2}\,\mu(dx)
=0
\]
Finally, the functions of the form $x \mapsto e^{-\|x\|^2} f(x)$, with $f \in E$, form a dense subset of $C(K)$. It follows that $\mu=0$, as desired.
\end{proof}

If $X$ is of negative type, then any rescaling $(X, t d)$ with $t>0$ is also of negative type. In particular, it follows that all rescalings of a metric space satisfying the assumptions of Theorem~\ref{thm:posdef} are positive definite.

A converse to this fact was established by Meckes: a metric space is of negative type if and only if all its rescalings are positive definite (compactness is not needed here), \cite[Theorem~3.3]{Meckes2013PositiveDefinite}.

In the proof above, Steps 1 and 2 are standard (Step 1 can be found, for instance, in \cite{schoenberg1938metric}). Step 3—the observation that $\langle \cdot, \cdot \rangle_{\calW}$ is positive definite on the whole of $\calM(X)$, and not only on $\calM_{\mathrm{atom}}(X)$—appears to be new. It follows that $\langle \cdot, \cdot \rangle_{\calW}$ defines an inner product on $\calM(X)$, whose associated norm we denote by $\|\cdot\|_{\calW}$. This leads to a simplified proof of the uniqueness of the diversity-maximizing measure.

\begin{cor}
  A compact metric space of negative type has a unique diversity-maximizing measure.
\end{cor}

\begin{proof}
Let $\mu$ and $\nu$ be two diversity-maximizing measures $\mu$ and $\nu$. Consider the function $\varphi \colon t \mapsto \|\mu + t(\nu - \mu)\|_{\calW}^2$. This is a polynomial in $t$ of degree at most $2$. Moreover, for $t \in [0,1]$, the measure $\mu + t(\nu - \mu)$ is a probability measure. Hence the restriction $\varphi_{\mid [0,1]}$ attains its minimum at both endpoints $t=0$ and $t=1$. Moreover, $\varphi \geq 0$ everywhere. The only possibility $\varphi$ is a constant quadratic polynomial. Thus, the $t^2$ coefficient $\| \nu - \mu \|_\calW^2$ vanishes and this implies that $\mu = \nu$. 
\end{proof}

\subsection{Magnitude and its relationship with maximum diversity}
\label{subsec:magnitude_and_diversity}

We present it here only in the setting of compact metric spaces of negative type, where it can be viewed as a modification of the notion of diversity.

It is obtained by replacing probability measures in the maximization problem~\eqref{eq:def_diversity} with signed measures: we allow parts of $X$ to be charged negatively. The magnitude, denoted $|X|$, is therefore defined by
\begin{equation}
\label{eq:def_magnitude}
|X| = \sup \left\{ \frac{1}{\| \mu \|_\calW^2} : \mu \in \calM(X) \text{ and } \mu(X) = 1\right\}
\end{equation}
Unlike the case of maximum diversity, the supremum is not necessarily attained (although, if it is attained, the measure $\mu$ is unique, as can be shown in the same way as in Corollary~1). It is a norm minimization problem in the pre-Hilbert space $(\calM(X), \| \cdot \|_\calW)$. One difficulty is that this space is never complete when $X$ is infinite.

If there exists a measure $\mu \in \calM(X)$ with $\mu(X) = 1$ that minimizes $\| \mu \|_\calW^2$, then $X$ is said to be \emph{weighted}. If, moreover, this measure is a probability measure, then $X$ is said to be \emph{positively weighted}; in that case, magnitude and maximum diversity coincide. The following proposition provides a criterion for $X$ to be weighted, see \cite[Definition~5.4]{LeinsterRoff2021}.

\begin{prop}
\label{prop:weighted}
$X$ is weighted if and only if there is a measure $w \in \calM(X)$ such that $Zw = \ind_X$. In this case, one has $|X| = w(X)$, and the supremum in~\eqref{eq:def_magnitude} is attained at $\mu = w / w(X)$.
\end{prop}

\begin{proof}
Suppose that the supremum in~\eqref{eq:def_magnitude} is attained at $\mu$.
For all $x, y \in X$ and $t \in \R$, we have
\[
\left\|\mu + t(\delta_y - \delta_x) \right\|_\calW^2  
= \| \mu \|_\calW^2 + 2t \langle\mu, \delta_y - \delta_x \rangle_\calW + o(t)
\geq \| \mu \|_\calW^2
\]
which forces $\langle\mu, \delta_y - \delta_x\rangle_\calW = Z\mu(y) - Z\mu(x) = 0$. In other words, the function $Z\mu$ is constant, and necessarily $Z\mu = 1/|X|$. It then suffices to set $w = |X| \mu$.

Conversely, suppose there exists a measure $w \in \calM(X)$ such that $Zw = \ind _X$. Then $w \neq 0$ and
\[
\| w \|_\calW^2 = \int_X Zw \, dw = w(X)
\]
so in particular $w(X) \neq 0$. Define $\mu = w/w(X) \in \calM(X)$ and note that $\mu(X) = 1$.

Let $\nu \in \calM(X)$ such that $\nu(X) = 1$. By the Cauchy--Schwarz inequality,
$\langle \mu, \nu \rangle_\calW^2 \leq \| \mu \|_\calW^2 \, \| \nu \|_\calW^2$.
Since $Z\mu = 1 / w(X)$ on $X$, we have
\[
\langle \mu, \nu \rangle_\calW = \int_X Z\mu \, d\nu = \frac{1}{w(X)}
\qquad\text{ and }\qquad
\| \mu \|_\calW^2 = \frac{1}{w(X)}
\]
It follows that
\[
\frac{1}{w(X)^2} \leq \frac{1}{w(X)} \, \| \nu \|_\calW^2
\]
Equivalently,
\[
\frac{1}{\| \nu \|_\calW^2} \leq w(X)
\]
with equality if and only if $\nu = \mu$. Therefore $|X| = w(X)$, and $\mu$ attains the supremum in~\eqref{eq:def_magnitude}.
\end{proof}

The measure $w$ is called the \emph{weight} of $X$. The interval $[0,L]$ provides an example of a positively weighted metric space of negative type: one checks by a direct computation that
\[
w = \frac{\delta_0 + \delta_L + \calL^1}{2}	
\]
where $\calL^1$ denotes Lebesgue measure, satisfies the equation $Zw = \ind_{[0, L]}$. Consequently, the magnitude of an interval of length $L$ is $1 + L/2$.

We will see in the next paragraph that all finite metric spaces (of negative type) are weighted, and later that the same holds for simplicial trees. Apart from these cases, however, most spaces are not weighted. Meckes showed that, in general, the appropriate space in which to look for a ``weight'' is the completion of $(\calM(X), \| \cdot \|_\calW)$, although this is no longer a space of measures; see \cite[Section~3]{Meckes2015Magnitude}.

When $X$ is finite, the linear system $Zw = \ind_X$ always admits a solution, since $Z$ is symmetric and positive definite. Here, $Z$ is viewed as a square matrix, even though the elements of $X$ are not ordered, and $w$ and $\ind_X$ as column vectors. By Proposition~\ref{prop:weighted}, the magnitude of $X$ is
\[
|X| = \sum_{x \in X} w(x) = \sum_{x, y \in X} Z^{-1}(x, y)
\]
where $Z^{-1}$ denotes the inverse matrix.

In particular, if $X = \{x,y\}$ is a two-point metric space in which the distance between $x$ and $y$ is $L$, then $X$ is clearly of negative type and
\[
Z =
\begin{pmatrix}
1 & e^{-L} \\
e^{-L} & 1
\end{pmatrix}, \qquad
Z^{-1} = \frac{1}{1 - e^{-2L}}
\begin{pmatrix}
1 & -e^{-L} \\
-e^{-L} & 1
\end{pmatrix}
\]
It follows that $|X| = 1 + \tanh (L/2)$. The magnitude ranges between $1$ and $2$, and it matches our interpretation of magnitude as the effective number of points in $X$.

The next proposition is an analogue of Proposition~\ref{prop:weighted} for maximum diversity.

\begin{prop}
\label{prop:diversity-max-meas}
 Let $X$ be a compact metric space of negative type. A probability measure $\mu \in \calP(X)$ is diversity-maximizing if and only if there is a constant $C > 0$ such that
 \begin{equation}
 \label{eq:cond_diversity}
 Z \mu = C \text{ on } \rmsupp \mu \qquad \text{ and } \qquad Z \mu \geq C \text{ on } X
 \end{equation}
 In that case, $C = 1 / |X|_+$.
\end{prop}

\begin{proof}
Suppose that $\mu$ is diversity-maximizing. For every $y \in X$ and every $t \in [0,1]$, the probability measure $(1-t)\mu + t\delta_y$ performs no better than $\mu$, that is,
\[
\left\|(1-t)\mu + t\delta_y \right\|_\calW^2 
= \| \mu \|_\calW^2 + 2t \langle\mu,\delta_y - \mu\rangle_\calW + o(t)
\geq \| \mu \|_\calW^2
\]
It follows that $\langle\mu,\delta_y - \mu\rangle_\calW = Z\mu(y) - \int_X Z\mu\, d\mu \geq 0$. Setting $C = \int_X Z\mu\, d\mu > 0$, which is the mean value of $Z\mu$, and using the continuity of $Z\mu$, we obtain~\eqref{eq:cond_diversity}. It is then clear that $C = 1/|X|_+$.

Conversely, assume that $\mu$ satisfies~\eqref{eq:cond_diversity}, and let $\nu \in \calP(X)$ be a competing measure. We immediately have
\[
\langle\mu,\nu\rangle_\calW = \int_X Z\mu \, d\nu \geq C
\]
and similarly $\langle\mu,\mu\rangle_\calW = C$. By the Cauchy--Schwarz inequality,
\[
C^2 \leq \langle\mu,\nu\rangle_\calW^2 \leq \|\mu\|_\calW^2 \| \nu \|_\calW^2 = C \| \nu \|_\calW^2
\]
so that $\|\nu \|_\calW^2 \geq C = \| \mu \|_\calW^2$. Hence $\mu$ is optimal.
\end{proof}

A somewhat counterintuitive fact is that a diversity-maximizing measure need not have full support ($\rmsupp \mu = X$). When it does, then by Propositions~\ref{prop:weighted} and~\ref{prop:diversity-max-meas} the space $X$ is positively weighted, a situation that is rather exceptional. This implies that, in biodiversity maximization problems, some species must necessarily be excluded; see \cite[Section~11]{leinster_meckes_diversity} for a discussion of this phenomenon.

The following provides a link between magnitude and maximum diversity:

\begin{thm}
Let $X$ be a compact metric space of negative type. One has
\[
|X|_+ = \sup \left\{ |Y| : Y \subseteq X \text{ is closed and positively weighted} \right\}
\]
and this supremum is attained when $Y$ is the support of the diversity-maximizing measure of $X$.
\end{thm}

\begin{proof}
The inequality $\geq$ is immediate: for every closed and positively weighted subset $Y \subseteq X$, one has
$|X|_+ \geq |Y|_+ = |Y|$.

For the reverse inequality, let $\mu$ be a diversity-maximizing measure on $X$, and set $Y = \rmsupp \mu$. By Proposition~\ref{prop:diversity-max-meas}, we have $Z\mu = 1/|X|_+$ on $Y$, which implies that
\[
\int_Y Z(x,y)\,\mu(dy) = \frac{1}{|X|_+} \qquad \text{for all } x \in Y
\]
Applying Proposition~\ref{prop:weighted} to $Y$ (instead of $X$), we deduce that the normalized restricted measure $|X|_+\,\mu_{\mid Y} \in \calM(Y)$ is a weight for $Y$. Since this measure is nonnegative, $Y$ is positively weighted, and
$|Y| = (|X|_+\,\mu_{\mid Y})(Y) = |X|_+$.
\end{proof}

When $X$ is finite, the previous theorem is known as the \emph{computation theorem} \cite[Theorem~7.1]{leinster_meckes_diversity}, since it provides an algorithm for computing $|X|_+$. One simply enumerates all subsets $Y \subseteq X$ and solves the linear system $Z_Y w_Y = \ind_Y$ for each of them (where $Z_Y$ denotes  the similarity kernel of $Y \subseteq X$). Then
\[
|X| = \max \left\{ \sum_{y \in Y} w_Y(y) : Y \subseteq X \text{ and } w_Y \geq 0 \right\}
\]
Unfortunately, this method has exponential complexity. This stems from the inherently nonlinear aspect of the problem, namely the determination of the support $Y$ of a diversity-maximizing measure.

\section{Magnitude of trees}
\label{sec:magtree}

\subsection{Weighted trees}
\label{subsec:wtree}

We begin by describing an operation with respect to which magnitude behaves well (unlike maximum diversity).
\begin{defi}
Let $(X,d_X)$ and $(Y,d_Y)$ be metric spaces such that $X \cap Y = \{a\}$. The \emph{wedge sum} of $X$ and $Y$ at $a$ is the space $X \cup Y$ endowed with the metric
\[
d(x,y)=
\begin{cases}
d_X(x,y) & \text{if } x,y \in X\\
d_Y(x,y) & \text{if } x,y \in Y\\
d_X(x,a)+d_Y(a,y) & \text{if } x \in X,\ y \in Y\\
d_X(y,a)+d_Y(a,x) & \text{if } x \in Y,\ y \in X
\end{cases}
\]
\end{defi}
The following result is likely to be known; part (b) is just a continuous analogue of \cite[Proposition~2.3.2]{Leinster2013Magnitude}.
\begin{prop}
\label{prop:wedge_sum_of_neg}
Let $X$ and $Y$ be compact metric spaces of negative type such that $X \cap Y = \{a\}$.
\begin{itemize}
\item[\upshape{(a)}] The wedge sum of $X$ and $Y$ at $a$ is of negative type.
\item[\upshape{(b)}] If $X$ and $Y$ are weighted, with weights $w_X$ and $w_Y$, then their wedge sum is weighted with weight
\[
w = \iota_{X\#} w_X + \iota_{Y\#} w_Y - \delta_a
\]
where $\iota_X$ and $\iota_Y$ denote the inclusion maps of $X$ and $Y$ into $X \cup Y$. Consequently,
\[
|X \cup Y| = |X| + |Y| - 1
\]
\end{itemize}
\end{prop}

\begin{proof}
(a). Let $h_X : (X,\sqrt{d_X}) \to H$ and $h_Y : (Y,\sqrt{d_Y}) \to H'$ be isometric embeddings into Hilbert spaces $H$ and $H'$. Let $H \oplus_2 H'$ be the Hilbert space equipped with the norm $\|(u,v)\| = \left(\|u\|_H^2 + \|v\|_{H'}^2\right)^{1/2}$. One readily checks that the map
\[
x \mapsto
\begin{cases}
(h_X(x),h_Y(a)) & \text{if } x \in X\\
(h_X(a),h_Y(x)) & \text{if } x \in Y
\end{cases}
\]
is an isometric embedding of $(X \cup Y,\sqrt{d})$ into $H \oplus_2 H'$.

(b). This follows by a direct computation. Let $x \in X \cup Y$; without loss of generality, assume $x \in X$. Then
\begin{align*}
Zw(x)
&= \int_{X \cup Y} e^{-d(x,y)}\, w(dy) \\
&= \int_X e^{-d(x,y)}\, w(dy)
   + \int_Y e^{-d(x,y)}\, w(dy)
   - w(\{a\}) e^{-d(x,a)}
\end{align*}
The integral over $X$ is
\begin{align*}
\int_X e^{-d(x,y)}\, w(dy)
&= \int_X e^{-d_X(x,y)}\, w_X(dy)
  + \int_X e^{-d(x,y)}\, \iota_{Y\#} w_Y(dy)
  - e^{-d(x,a)} \\
&= 1 + e^{-d(x,a)} w_Y(\{a\}) - e^{-d(x,a)}
\end{align*}
while the integral over $Y$ is
\begin{align*}
\int_Y e^{-d(x,y)}\, w(dy)
&= \int_Y e^{-d(x,y)}\, \iota_{X\#} w_X(dy)
  + \int_Y e^{-d(x,y)}\, w_Y(dy)
  - e^{-d(x,a)} \\
&= e^{-d(x,a)} w_X(\{a\})
  + e^{-d(x,a)} \int_Y e^{-d_Y(a,y)}\, w_Y(dy)
  - e^{-d(x,a)} \\
&= e^{-d(x,a)} w_X(\{a\}).
\end{align*}
Combining these identities yields $Zw(x)=1$ on $X$, and by symmetry also on $Y$.
\end{proof}

\begin{defi}
A \emph{weighted tree} is a triple $(X,E,\ell)$ where
\begin{itemize}
\item $(X,E)$ is a finite simple undirected graph; that is, $X$ is a finite set of \emph{vertices}, and $E \subseteq \{\{x,y\} : x,y \in X,\ x \neq y\}$ is the set of \emph{edges};
\item as a graph, $(X,E)$ is a tree, i.e. it is connected and has no cycles;
\item $\ell \colon E \to \R^{>0}$ assigns to each edge $e \in E$ its \emph{length}.
\end{itemize}
\end{defi}
A typical edge will be denoted by $e$. This conflicts with the notation for the exponential function, but no confusion should arise.
The \emph{degree} of a vertex $x \in X$ is the number of its neighbours, $\deg x = \#\{ y \in X : \{x,y\} \in E \}$.
Vertices of degree one are called \emph{leaves}, and vertices of degree at least three are called \emph{branch points}. In the case where $X$ has a single vertex, this vertex has degree~$0$ and is also regarded as a leaf.

One can equip $X$ with a metric structure by defining, for all vertices $x,y \in X$, the distance $d(x,y)$ as the sum of the lengths of the edges along the unique simple path connecting $x$ to $y$. In other words, if $e_1,\dots,e_p$ are the edges of this path, we set $d(x,y) = \ell(e_1) + \cdots + \ell(e_p)$.

Starting from a single point, a weighted tree can be constructed by iteratively performing wedge sums of two-point spaces. This simple observation is at the heart of the following result. The observation that weighted tree are of negative type was already made in~\cite[Theorem~3.6]{Meckes2013PositiveDefinite}.
\begin{prop}
\label{prop:magnitude_weighted_tree}
Let $X$ be a weighted tree. Then $X$ is of negative type and its weight (solution of $Zw = \ind_X$) is given by
\[
w(x) = \sum_{e \ni x} \frac{1}{1 + e^{-\ell(e)}} - (\deg x - 1), \qquad x \in X
\]
The magnitude of $|X|$ is
\[
|X| = 1 + \sum_{e \in E} \tanh \left( \frac{\ell(e)}{2} \right)
\]
\end{prop}
\begin{proof}
The result is proved by induction on $\#X$. The case where $X$ consists of a single vertex is immediate. Assume now that $\#X \geq 2$, and choose a leaf $p \in X$, whose unique neighbour is denoted by $q \in X$.

As a metric space, $X$ is the wedge sum of the weighted tree $X \setminus \{p\}$ and $\{p,q\}$ at the point $q$. By induction and Proposition~\ref{prop:wedge_sum_of_neg}, $X$ is of negative type. Recall that the magnitude of the two-point space $\{p,q\}$ is $1 + \tanh(d(p,q)/2)$, and that its weight assigns the value $(1 + e^{-d(p,q)})^{-1}$ to both $p$ and $q$.
Let $w_{X \setminus \{p\}}$ denote the weight of $X \setminus \{p\}$. By Proposition~\ref{prop:wedge_sum_of_neg}, the weight $w$ of $X$ is given by
\[
w(x) =
\begin{cases}
w_{X \setminus \{p\}}(x) & \text{if } x \in X \setminus \{p,q\} \\
(1 + e^{-d(p,q)})^{-1} & \text{if } x = p \\
w_{X \setminus \{p\}}(q) + (1 + e^{-d(p,q)})^{-1} - 1 & \text{if } x = q
\end{cases}
\]
This yields, by induction, the desired formula for $w$.
\end{proof}

\subsection{Simplicial metric trees}
\label{subsec:simplicial}

A \emph{segment} $S$ in a metric space $X$ is a subset isometric to a non-empty compact interval of $\mathbb{R}$. An \emph{endpoint} of $S$ is a point $x \in S$ such that $S \setminus \{x\}$ is connected; there are exactly two such points, unless $S$ consists of a single point.
A metric space $X$ is said to be \emph{uniquely geodesic} if, for any two points $x,y \in X$, there exists a unique segment whose set of endpoints is $\{x,y\}$. This segment is then denoted by $\lseg x,y \rseg$. We write $\rseg x,y \lseg = \lseg x,y \rseg \setminus \{x,y\}$, and similarly introduce the notations $\lseg x,y \lseg$ and $\rseg x,y \rseg$ in the obvious way. A subset $Y \subseteq X$ of a uniquely geodesic metric space is said to be \emph{geodesically convex} whenever $\lseg x, y \rseg \subseteq Y$ for all $x, y \in Y$.
\begin{defi}
\label{def:simplicial_tree}
A \emph{simplicial (metric) tree} is a metric space $X$ obtained by gluing finitely many geodesic segments. More precisely, $X$ can be constructed in finitely many steps as
$X^{(0)} \subseteq X^{(1)} \subseteq \cdots \subseteq X^{(p)} = X$, so that
\begin{itemize}
\item $X^{(0)}$ is a singleton;
\item for each $k \in \{0,\dots,p-1\}$, there exists a point $a \in X^{(k)}$ and a segment $S$, one of whose endpoints is $a$, such that
$S \subseteq (X \setminus X^{(k)}) \cup \{a\}$ and $X^{(k+1)}$ is the wedge sum of $X^{(k)}$ and $S$ at $a$.
\end{itemize}
\end{defi}
For any $a \in X$, the \emph{degree} of $a$, denoted $\deg(a)$, is the number of connected components of $X \setminus \{a\}$. This number is always finite, and when it is different from $2$, the point $a$ is called a \emph{vertex}. The vertices necessarily appear in the construction of $X$ at the second step of Definition~\ref{def:simplicial_tree}; as a consequence, $X$ has finitely many vertices. Vertices of degree $1$ are called \emph{leaves}, and those of degree at least $3$ are called \emph{branch points}. In the degenerate case where $X$ consists of a single point, this point has degree $0$ but is still regarded as a leaf.

One can show that the wedge sum of two uniquely geodesic metric spaces is again uniquely geodesic. In view of the inductive construction of simplicial trees, it follows that simplicial trees are uniquely geodesic. Moreover, by Proposition~\ref{prop:wedge_sum_of_neg}, they are compact and of negative type.

We recall from Subsection~\ref{subsec:magnitude_and_diversity} that a geodesic segment $\lseg x,y \rseg$ is positively weighted, with weight $(\delta_x + \delta_y + \calH^1)/2$, and has magnitude $1 + d(x,y)/2$, where $\calH^1$ denotes the one-dimensional Hausdorff measure. The following proposition presents no difficulty and is proved along the same lines as Proposition~\ref{prop:magnitude_weighted_tree}. Note in particular the special role played by branch points, which prevents $X$ from being positively weighted.
\begin{prop}
\label{prop:magnitude_simplicial}
 Every simplicial tree $X$ is weighted, with weight
 \[
 w = \frac{\calH^1}{2} + \sum_{x \in X} \left( 1 - \frac{\deg x}{2}\right) \delta_x
 \]
 and magnitude $|X| = 1 + \calH^1(X)/2$.
\end{prop}

\subsection{$\R$-trees}
\label{subsec:Rtree}
We first give the definition of $\R$-trees, together with a brief overview of their basic geometric properties. The interested reader may consult \cite{evans2008probability} for further details.

\begin{defi}
 An \emph{$\R$-tree} is a metric space such that
 \begin{itemize}
 \item $X$ is uniquely geodesic;
 \item for all $x, y, z \in X$, one has $\lseg x, z \rseg \subseteq \lseg x, y \rseg \cup \lseg y, z\rseg$.
 \end{itemize}
\end{defi}
There are many equivalent definitions of $\R$-trees. Very often, the second condition is replaced by ``for every continuous injective map $\gamma \colon [0, 1] \to X$, one has $\gamma([0, 1]) = \lseg x, y \rseg$''. The degree of a point $x \in X$ is defined as for simplicial trees, as the number of connected components of $X \setminus \{x\}$. Vertices of degree~1 are called leaves, those of degree $\geq 3$ are branch points. One novelty is that points of degree~2 are called \emph{regular points}, as is common in the literature around $\R$-trees. The set of leaves and branch points are respectively denoted $\operatorname{Leaves}(X)$ and $\operatorname{Branch}(X)$. One can check that the wedge sum of two $\R$-trees is an $\R$-tree. This implies that simplicial trees are $\R$-trees.

The \emph{skeleton} of an $\R$-tree is
\[
\operatorname{Sk}(X) = \bigcup_{x, y \in X} \rseg x, y \lseg
\]
If $D$ is a dense subset of $X$, then
\begin{equation}
\label{eq:skeletoncountable}
\operatorname{Sk}(X) = \bigcup_{x, y \in D} \rseg x, y \lseg
\end{equation}
Indeed, for any $x, y \in X$ and $z \in \mathopen{\rseg} x, y \mathclose{\lseg}$, one can find point $x', y' \in D$ such that $d(x, x') < d(x, z)$ and $d(y, y') < d(y, z)$. This implies that $z \not\in \mathopen{\lseg} x, x' \mathclose{\rseg}$ and $z \not\in \mathopen{\lseg} y, y'\mathclose{\rseg}$. However, $z \in \mathopen{\lseg} x, y\mathclose{\rseg} \subseteq \lseg x, x'\rseg \cup \lseg x', y'\rseg \cup \lseg y', y\rseg$. Consequently, $z \in \mathopen{\rseg} x', y' \mathclose{\lseg}$.

We will be interested only in compact $\R$-trees. They are in particular separable, so we can find a countable dense subset $D$. This entails that the skeleton is a Borel subset of $X$, that has Hausdorff dimension equal to~1. The Borel measure $\lambda \colon A \mapsto \calH^1(\operatorname{Sk}(X) \cap A)$ is called the \emph{length measure}. It is satisfies the property $\lambda(\lseg x, y \rseg) = d(x, y)$ for all $x, y \in X$. The value $\lambda(X) = \calH^1(\operatorname{Sk}(X)) \in [0, \infty]$ is the \emph{total length} of $X$. (Note that the length measure makes sense even when $\operatorname{Sk}(X)$ is not a Borel subset of $X$, as $\calH^1$ is an outer measure. It can be proved that $\lambda$ is a Borel measure in any case).

For the reader’s convenience, we collect in the following proposition a number of elementary facts about $\R$-trees, which are well known but not always easy to locate in the literature, and which will be used throughout the paper.
\begin{prop}
\label{prop:facts_Rtree}
Let $X$ be an $\R$-tree.
\begin{enumerate}
\item[\upshape{(a)}] for any $x, y, z \in X$, if $\lseg x, y \rseg \cap \lseg y, z \rseg = \{y\}$, then $y \in \lseg x, z \rseg$;
\item[\upshape{(b)}] (tripod configuration) for any $x, y, z \in X$, there is a unique point $t \in X$ such that $\lseg x, y \rseg \cap \lseg x, z\rseg = \lseg x, t \rseg$. Moreover, $t \in \lseg y, z\rseg$;
\item[\upshape{(c)}] $X$ is of negative type;
\item[\upshape{(d)}] for all $x, y \in X$ and $z \in \mathopen{\rseg} x, y \mathclose{\lseg}$, the points $x, y$ belong to different components of $X \setminus \{z\}$, hence $\deg z \geq 2$;

\item[\upshape{(e)}] $\operatorname{Sk}(X) = X \setminus \operatorname{Leaves}(X)$.
\end{enumerate}
\end{prop}
\begin{proof}
(a). By definition of an $\R$-tree, we have the inclusion $\lseg x, z \rseg \subseteq \lseg x, y \rseg \cup \lseg y, z \rseg$. Also, by assumption, $\lseg x, y \rseg \cup \lseg y, z\rseg$ is a topological segment, meaning that it is homeomorphic to a nonempty compact interval of $\R$. The endpoints of $\lseg x, y \rseg \cup \lseg y, z \rseg$, namely the points $p$ such that $\lseg x, y \rseg \cup \lseg y, z \rseg \setminus \{p\}$ is connected, are $x$ and $z$.
As $\lseg x, z\rseg$ is also a topological segment and has the same endpoints, it is now a purely topological fact that $\lseg x, z \rseg = \lseg x, y \rseg \cup \lseg y, z\rseg$, meaning that $y \in \lseg x, z\rseg$.

(b). The set $\lseg x, y\rseg \cap \lseg x, z\rseg$ is a closed and geodesically convex subset of the segment $\lseg x, y \rseg$ that contains $x$. Thus it has the form $\lseg x, t \rseg$. The point $t$ is uniquely determined by $x, y, z$ as an endpoint of $\lseg x, y\rseg \cap \lseg x, z\rseg$. 

Moreover, $\lseg y, t \rseg \cap \lseg t, z \rseg = \{t\}$. Indeed, if $\lseg y, t \rseg \cap \lseg t, z \rseg$ contains a point $t'$ other than $t$, one easily infers the contradiction
\[
\lseg x, t \rseg \subsetneq \lseg x, t' \rseg \subseteq \lseg x, y \rseg \cap \lseg x, z \rseg = \lseg x, t \rseg
\]
We can now deduce from~(a) that $t \in \lseg y, z \rseg$.

(c). Fix a point $x \in X$ (that we may regard as a root). It follows from~(a) that $\lambda$-almost everywhere, $|\ind_{\lseg x, y \rseg} - \ind_{\lseg x, z \rseg}| = \ind_{\lseg y, z \rseg}$ for all $y, z \in X$. Thus
\[
\left( \int_X |\ind_{\lseg x, y \rseg} - \ind_{\lseg x, z \rseg} |^2\, d\lambda \right)^{1/2} = \sqrt{d(y, z)}
\]
Put differently, the map $y \mapsto \ind_{\lseg x, y\rseg}$ is an isometric embedding of $(X, \sqrt{d})$ into the Hilbert space $L^2(X, \lambda)$.

%

(d). Suppose, for contradiction, that $x$ and $y$ lie in the same connected component $C$ of $X \setminus \{z\}$. Since $X$ is uniquely geodesic, the space $X \setminus \{z\}$ is locally path-connected; hence $C$ is open and path-connected. In particular, there exists a continuous path $\gamma \colon [0,1] \to C$ such that $\gamma(0)=x$ and $\gamma(1)=y$.

Because $C$ is open and $\gamma([0,1])$ is compact, there exists $\varepsilon>0$ such that for every $t \in [0,1]$, the closed ball $B(\gamma(t), \varepsilon)$ is contained in $C$. By uniform continuity of $\gamma$, there exists $\delta>0$ such that whenever $|t'-t|<\delta$, one has $d(\gamma(t),\gamma(t'))<\varepsilon$.
Choose a subdivision $0=t_0 < t_1 < \cdots < t_n=1$ with $\max_i (t_{i+1}-t_i) \le \delta$. Then for each $i$, the segment $\lseg \gamma(t_i), \gamma(t_{i+1}) \rseg$ is contained in $C$, and therefore
\[
\lseg x, y \rseg \subseteq \bigcup_{i=0}^{n-1} \mathopen{\lseg} \gamma(t_i), \gamma(t_{i+1}) \mathclose{\rseg} \subseteq C
\]
This implies $z \in \lseg x, y \rseg \subseteq C$, contradicting the fact that $z \not\in C$.

(e). The inclusion $\subseteq$ follows from~(d). Conversely, suppose $x \in X$ is not a leaf, so that $X \setminus \{x\}$ has at least two connected components $C_1$ and $C_2$. Let $y_1 \in C_1$ and $y_2 \in C_2$. The sets $\rseg x, y_1\rseg$ and $\rseg x, y_2 \rseg$ are connected subsets of $X \setminus \{x\}$, so they are respectively in the components $C_1$ and $C_2$. It then follows that $x \in \mathopen{\rseg} y_1, y_2 \mathclose{\lseg}$ by~(a), hence $x \in \operatorname{Sk}(X)$.
\end{proof}

In (b), if we think of $x$ as being the root of $X$, then the point $t$ is the lowest ancestor of $y$ and $z$. We briefly comment on~(f). In contrast with the simplicial case, the skeleton $\operatorname{Sk}(X)$ of an $\R$-tree may be negligible compared to $\operatorname{Leaves}(X)$. In particular, $\operatorname{Leaves}(X)$ can have Hausdorff dimension strictly greater than~$1$, a phenomenon that is typical for random $\R$-trees. For example, the Brownian tree described in~\cite{legall2005random} has, almost surely, infinite total length and Hausdorff dimension equal to~$2$. Roughly speaking, one claim of this paper is that the magnitude of an $\mathbb{R}$-tree is governed by the size of its skeleton, whereas its diversity is more closely tied to the size of its leaves.

We now proceed to computing the magnitude of compact $\R$-trees.

\begin{thm}
\label{thm:magnitude_Rtree}
Let $X$ be a compact $\mathbb{R}$-tree. Then
\[
|X| = 1 + \frac{\lambda(X)}{2}.
\]
\end{thm}

\begin{proof}
Fix a point $\rho \in X$, which we may regard as the root of $X$, and let $(x_n)_{n \geq 1}$ be a dense sequence in $X$. Then
\[
\bigcup_{n=1}^\infty \mathopen{\lseg}\rho, x_n \mathclose{\lseg}
= \{\rho\} \cup \operatorname{Sk}(X)
\]
(For the inclusion $\supset$, note that we have $\mathopen{\rseg} x_n, x_m \mathclose{\lseg} \subseteq \mathopen{\lseg} \rho, x_n \mathclose{\lseg} \cup \mathopen{\lseg} \rho, x_m \mathclose{\lseg}$ and $\operatorname{Sk}(X) = \bigcup_{n, m} \rseg x_n, x_m \lseg$.)

Set $X^{(0)} = \{\rho\}$ and define inductively $X^{(n+1)} = X^{(n)} \cup \lseg \rho, x_{n+1} \rseg$.
We claim that each $X^{(n)}$ is a simplicial tree. This is clear for $n=0$. Assume that $X^{(n)}$ is a simplicial tree. If $X^{(n+1)} = X^{(n)}$, there is nothing to prove in this case. Otherwise, since $X^{(n)}$ is a simplicial tree, it is a closed and geodesically convex part of $X$. Hence, $X^{(n)} \cap \lseg \rho, x_{n+1}\rseg$ is a closed and geodesically convex subset of the segment $\lseg \rho, x_{n+1} \rseg$ that contains $\rho$. Thus it has the form $\lseg \rho, a\rseg$, for some $a \in X^{(n)}$.

We claim that $X^{(n+1)}$ is the wedge sum of $X^{(n)}$ and $\lseg a, x_{n+1} \rseg$ at $a$. First we have
\[
\{a\} \subseteq X^{(n)} \cap \lseg a, x_{n+1} \rseg \subseteq X^{(n)} \cap \lseg \rho, x_{n+1} \rseg \cap \lseg a, x_{n+1} \rseg \subseteq
\lseg \rho, a \rseg \cap \lseg a, x_{n+1} \rseg = \{a \}
\]
thus $X^{(n)}$ and $\lseg a, x_{n+1}\rseg$ intersect only at $a$. Now, fix $x \in X^{(n)}$ and $y \in \lseg a, x_{n+1} \rseg$. Then $\lseg x, a\rseg \cap \lseg a, y\rseg = \{a\}$ because $\lseg x, a\rseg \subseteq X^{(n)}$ and $\lseg a, y \rseg \subseteq \lseg a, x_{n+1}\rseg$. Therefore, by Proposition~\ref{prop:facts_Rtree}(a), we have $a \in \lseg x, y \rseg$. It follows that $d(x, y) = d(x, a) + d(a, y)$ and the claim is proved.

By Proposition~\ref{prop:magnitude_simplicial}, it follows that
\[
|X^{(n)}|
= 1 + \frac{1}{2}\calH^1\!\left(\bigcup_{k=1}^n \mathopen{\lseg}\rho, x_k \mathclose{\lseg}\right)
\]
Since $|X| \geq |X^{(n)}|$ for all $n$, we obtain by letting $n \to \infty$
\begin{equation}
\label{eq:halfgeq}
|X| \geq 1 + \frac{1}{2}\calH^1\!\left(\bigcup_{n=1}^\infty \mathopen{\lseg}\rho, x_n \mathclose{\lseg}\right)
= 1 + \frac{\calH^1(\operatorname{Sk}(X))}{2} = 1 + \frac{\lambda(X)}{2}
\end{equation}

Moreover, each $X^{(n)}$ contains $\{x_1,\dots,x_n\}$, hence $X^{(n)} \to X$ in the Gromov--Hausdorff sense. By a result of Meckes \cite[Theorem~2.6]{Meckes2013PositiveDefinite}, magnitude is lower semicontinuous on the class of compact positive definite metric spaces with respect to Gromov--Hausdorff convergence. This yields the reverse inequality in~\eqref{eq:halfgeq} and completes the proof.
\end{proof}
In particular, if $X$ is such that $\lambda(X) = \infty$, then its magnitude is infinite. The question of whether there exist compact metric spaces of negative type with infinite magnitude was raised in \cite{Meckes2013PositiveDefinite, Meckes2015Magnitude, LeinsterMeckes2017156193}. A counterexample was constructed in \cite[Section~2]{MeckesLeinster2023ExtremalMagnitude}, which we briefly recall.

Let $(a_n)$ be a sequence of positive real numbers tending to $0$, sufficiently slowly so that $\sum_{n=1}^\infty a_n = \infty$. For each $n \in \mathbb{N}^*$, let $e_n \in \ell^1$ denote the sequence whose $n$-th coordinate is $1$ and all others are zero. Consider $X$ as the closure of the convex hull of $\{a_n e_n : n \in \mathbb{N}^*\}$ in $\ell^1$.

The compactness of $X$ follows from the fact that $a_n \to 0$, while the fact that $X$ is of negative type comes from its embedding into $\ell^1$. Leinster and Meckes show that $|X| = \infty$, but this also follows from the fact that $X$ contains
\[
\bigcup_{n=1}^\infty \{ t a_n e_n : t \in [0,1] \}
\]
which is a compact $\mathbb{R}$-tree of total length $\sum_{n=1}^\infty a_n = \infty$. Thus, all currently known counterexamples arise, in a certain sense, from trees of infinite total length.

We conclude this section with a comment on embeddings of compact $\R$-trees. One easily checks that, for fixed $x \in X$, the map $X \to L^1(X, \lambda)$ defined by $y \mapsto \ind_{\lseg x, y \rseg}$ is an isometric embedding. It is also a folklore result that compact $\R$-trees can be isometrically embedded into $\ell^1$. The following non-embedding result is not surprising, but its proof relies on the notion of magnitude, and we are not aware of any alternative proof.

\begin{prop}
A compact $\R$-tree of infinite length cannot be isometrically embedded in a finite-dimensional subset of $L^1$.
\end{prop}

\begin{proof}
By Theorem~\ref{thm:magnitude_Rtree}, such a tree has infinite magnitude. However, Leinster and Meckes proved that compact subsets of finite-dimensional subspaces of $L^1$ have finite magnitude, \cite[Theorem~3.1]{MeckesLeinster2023ExtremalMagnitude}.
\end{proof}

\section{Maximum diversity of weighted trees}
\label{sec:max_div_wtree}

\subsection{A polynomial-time algorithm for computing maximum diversity}

Let $(X, E, \ell)$ be a weighted tree. We recall that the similarity matrix $Z \colon X \times X \to \mathopen{(} 0, 1 \mathclose{]}$ is symmetric and positive definite. Before computing the maximum diversity, we need to understand the structure of $Z$ and $Z^{-1}$. There exists a centered Gaussian process $U = (U_x)_{x \in X}$ indexed by $X$, with covariance matrix $Z$. Working with this process will make the linear-algebraic manipulations that follow more transparent.

The density of $U$ is expressed in terms of the inverse matrix $Z^{-1}$ as
\begin{equation}
\label{eq:densiteU}
f_U(u) = \frac{1}{(2\pi)^{\# X / 2} \sqrt{\det Z}} \exp\left( - \frac{u^T Z^{-1} u}{2} \right),\qquad \text{for } u\in\R^X
\end{equation}
We will show that this density admits a particularly simple factorization in terms of the edges of $X$. This will reflect the fact that the matrix $Z^{-1}$ is sparse.

We begin by fixing some notation. For any subset $A \subseteq X$, we write $U_A = (U_x)_{x\in A}$ for the restriction of $U$ to $A$. If $X = A\cup B$ is a partition, we decompose the similarity matrix as
\begin{equation}
\label{eq:blocz}
Z = \begin{pmatrix}
Z_A & M^\top \\ 
M & Z_B
\end{pmatrix}
\end{equation}
so that $Z_A$ and $Z_B$ are the covariance matrices of $U_A$ and $U_B$, respectively. Recall the standard fact for centered Gaussian vectors: conditionally on $U_A = u_A$, the vector $U_B$ is Gaussian with mean $M Z_A^{-1} u_A$ and covariance matrix $Z_B - M Z_A^{-1} M^\top$.

Now assume $\# X \ge 2$, and fix a leaf $p \in X$, with unique parent denoted $q$. Let $e \in E$ be the edge connecting $p$ to $q$. We apply the previous result with $A = X \setminus \{p\}$ and $B = \{p\}$ to compute the conditional density of $U_p$ given $U_A = u_A$, denoted $f_{U_p}(\cdot \mid U_A = u_A)$.

In this setting, the row $M$ takes a particularly simple form:
\[
M = e^{-\ell(e)} Z_A(q, \cdot)
\]
since any path from the leaf $p$ to another vertex must pass through $q$. It follows that
\begin{gather*}
\operatorname{Var}(U_p \mid U_A = u_A)
= 1 - e^{-2\ell(e)} Z_A(q,\cdot)\, Z_A^{-1}\, Z_A(\cdot,q)
= 1 - e^{-2\ell(e)} Z_A(q,q)
= 1 - e^{-2\ell(e)} \\
\mathbb{E}(U_p \mid U_A = u_A)
= e^{-\ell(e)} Z_A(q,\cdot)\, Z_A^{-1} u_A
= e^{-\ell(e)} u_q
\end{gather*}

Thus we obtain a spatial Markov property: conditioning on $U_A$ provides no more information about $U_p$ than conditioning on $U_q$ alone.

We can now decompose the density of $U$ as
\begin{align*}
f_U(u)
&= f_{U_p}(u_p \mid U_A = u_A)\, f_{U_A}(u_A) \\
&= \frac{1}{\sqrt{2\pi(1 - e^{-2\ell(e)})}}
\exp\left(
- \frac{(u_p - e^{-\ell(e)} u_q)^2}{2(1 - e^{-2\ell(e)})}
\right)
f_{U_A}(u_A)
\end{align*}
where $f_{U_A}$ denotes the density of $U_A$.

Since $A$ is itself a weighted tree with one fewer leaf, this procedure can be iterated. Repeating the argument yields a full factorization of $f_U$ over the edges of the tree. The preceding discussion therefore shows, by induction, the following structural result for $Z^{-1}$. Alternatively, this result can be proved directly by induction using \cite[Corollary~2.3.3]{Leinster2013Magnitude}.

\begin{prop}
\label{prop:wtree_Z-1}
Let $x,y \in X$ with $x \neq y$. Then
\[
Z^{-1}(x,x) = 1 + \sum_{e \ni x} \frac{e^{-2\ell(e)}}{1 - e^{-2\ell(e)}}
\]
and
\[
Z^{-1}(x,y) = 0 \quad \text{if } x \text{ and } y \text{ are not adjacent}
\]
If $x$ and $y$ are adjacent and $e = \{x,y\}$ denotes the corresponding edge, then
\[
Z^{-1}(x,y) = - \frac{e^{-\ell(e)}}{1 - e^{-2\ell(e)}}
\]
\end{prop}

We present Algorithm~\ref{algomax} for computing the maximum diversity of a finite weighted tree $X$ with similarity matrix $Z$ in polynomial time. The problem of whether maximum diversity can be computed in polynomial time for classes of metric spaces was raised in \cite[Section~12]{leinster_meckes_diversity}. This principle of this algorithm was recently introduced by Huntsman in \cite[Section~3]{huntsman2025peeling} for a first-order approximation to the diversity maximization problem, in the more general setting of finite metric spaces satisfying a condition slightly stronger than negative type. Some of the arguments presented below are inspired by \cite{huntsman2025peeling}.

\begin{algorithm}
\begin{algorithmic}[1]
\Function{DiversityMaximizingMeasure}{$X$}
\State $Z \gets$ similarity kernel of $X$
\State $A \gets X$ 
\Repeat
    \State $w \gets 0$ \Comment{initialize $w$ to the $0$ function on $X$}
    \State $w_{\mid A} \gets Z^{-1}_A \ind_A$
    \State $A \gets \{x \in X : w(x) > 0\}$
\Until{$w \not\geq 0$}
\State $|X|_+ \gets \sum_{x \in X} w(x)$
\State $\mu \gets w / |X|_+$ \Comment{normalization}
\State \Return $\mu, |X|_+$
\EndFunction
\end{algorithmic}
\caption{Computing the diversity-maximizing measure}
\label{algomax}
\end{algorithm}

\begin{thm}
The algorithm \textsc{DiversityMaximizingMeasure} is correct when applied to a weighted tree.
\end{thm}

\begin{proof}
First, we show that at every step of the algorithm, the set $A$ is non-empty. This is clearly true at line~3. Now assume that $A$ is non-empty at the beginning of an iteration. Then the vector $w$, defined at lines~5--6, has at least one positive entry, since $Zw = 1$ on $A$ and all coefficients of $Z$ are positive. Hence $A$ remains non-empty at line~7.

At the end of the \textbf{repeat} loop, the vector $w$ is non-zero and non-negative. The normalization step (lines~9--10) is therefore well-defined, and the output $\mu$ is a probability measure on $X$.

To prove that $\mu$ is the desired solution, we use the criterion of Proposition~\ref{prop:diversity-max-meas}. We actually show by induction that, at every iteration (immediately after line~6),
\begin{equation}
\label{eq:condZw}
Zw = 1 \text{ on } \operatorname{supp}w = \{x \in X : w(x) \neq 0\}
\qquad \text{and} \qquad
Zw \geq 1 \text{ on } X
\end{equation}
After normalization (lines~9--10), this yields the required properties for $\mu$.

Assume that~\eqref{eq:condZw} holds at some stage of the algorithm. We then define (line~7) $A = \{x \in X : w(x) > 0\}$ and set $B = X \setminus A$. We write $Z$ in block form as in~\eqref{eq:blocz}:
\[
Z =
\begin{pmatrix}
Z_A & M^\top \\
M & Z_B
\end{pmatrix}
\]
By the induction hypothesis, we have
\begin{equation}
\label{eq:blocZw}
\begin{pmatrix}
Z_A & M^\top \\
M & Z_B
\end{pmatrix}
\begin{pmatrix}
w_{\mid A} \\
w_{\mid B}
\end{pmatrix}
=
\begin{pmatrix}
\ind_A \\
u
\end{pmatrix},
\qquad \text{with } u \geq \ind_B
\end{equation}
At the next iteration, the updated vector is
\[
w' =
\begin{pmatrix}
Z_A^{-1}\ind_A \\
0
\end{pmatrix}
\]
A direct computation gives
\[
Zw' = \begin{pmatrix}
        Z_A & M^\top \\ M & Z_B
    \end{pmatrix} \begin{pmatrix}
        Z_A^{-1} \ind_A \\ 0
    \end{pmatrix} = \begin{pmatrix}
        \ind_A \\ M Z^{-1}_A \ind_A
    \end{pmatrix}
\]
Hence $Zw' = 1$ on $A$ and $\operatorname{supp} w' \subseteq A$. It remains to prove that $M Z_A^{-1} \ind_A \geq \ind_B$.

From~\eqref{eq:blocZw}, we obtain
\begin{gather*}
M w_{\mid A} + M Z_A^{-1} M^\top w_{\mid B} = M Z_A^{-1} \ind_A \\
M w_{\mid A} + Z_B w_{\mid B} = u
\end{gather*}
One infers
\[
M Z_A^{-1} \ind_A
= u - (Z_B - M Z_A^{-1} M^\top) w_{\mid B}
\]
Since $u \geq \ind_B$ and $w_{\mid B} \leq 0$, it suffices to show that all coefficients of $Z_B - M Z_A^{-1} M^\top$ are non-negative.

Recall that $Z_B - M Z_A^{-1} M^\top$ is the conditional covariance matrix of $U_B$ given $U_A$. The density of $U$ is given by~\eqref{eq:densiteU}. And the conditional density of $U_B$ given $U_A = u_A$ is
\begin{align*}
u_B \mapsto C f_U(u_A,u_B)
&= C' \exp\left(-\frac{(u_A,u_B)^\top Z^{-1}(u_A,u_B)}{2}\right) \\
&= C'' \exp\left(
-\frac{1}{2} u_B^\top (Z_B - M Z_A^{-1} M^\top)^{-1} u_B
+ \text{affine terms in } u_B
\right)
\end{align*}
where
\[
C = \left( \int_{\R^B} f_U(u_A, u_B) \, du_B\right)^{-1}
\]
and $C', C''$ are other constants (that may depend on $u_A$).

By comparing the two expressions for the conditional density and using the structure of $Z^{-1}$ given in Proposition~\ref{prop:diversity-max-meas}, we observe that the diagonal entries of $(Z_B - M Z_A^{-1} M^\top)^{-1}$ are  positive, while its off-diagonal entries are non-positive. Moreover, both $Z_B - M Z_A^{-1} M^\top$ and its inverse are symmetric positive definite, hence all their eigenvalues are strictly positive.

It follows that $(Z_B - M Z_A^{-1} M^\top)^{-1}$ is a so-called $M$-matrix (see \cite[8.3.P15]{horn2012matrix}). We now recall why this implies that $Z_B - M Z_A^{-1} M^\top$ has non-negative entries. Let $m>0$ denote the largest diagonal entry of $(Z_B - M Z_A^{-1} M^\top)^{-1}$. We may write
\[
(Z_B - M Z_A^{-1} M^\top)^{-1} = mI - P = m\left(I - \frac{P}{m}\right)
\]
where $P$ is a matrix with non-negative entries and $I$ is the identity matrix. By a corollary of the Perron theorem (see \cite[Theorem~8.3.1]{horn2012matrix}), the spectral radius $\rho(P) = \max \{ |\lambda| : \lambda \text{ eigenvalue of } P\}$ is an eigenvalue of $P$. This implies that $m - \rho(P)$ is an eigenvalue of the positive definite matrix $(Z_B - M Z_A M^\top)^{-1}$, thus $m > \rho(P)$. Therefore $I - P/m$ is invertible with
\[
Z_B - M Z_A^{-1} M^\top
= \frac{1}{m}\left(I - \frac{P}{m}\right)^{-1}
= \frac{1}{m}\sum_{k=0}^\infty \left(\frac{P}{m}\right)^k
\]
This expansion shows that all entries of $Z_B - M Z_A^{-1} M^\top$ are non-negative, completing the proof.

Finally, note as a side remark that the conditional covariance matrix $Z_B - M Z_A^{-1} M^\top$ is itself sparse. Using the factorization of the density $f_U$ over edges, one can show that for any $x,y \in B$, the variables $U_x$ and $U_y$ are conditionally independent given $U_A$ whenever $A$ separates $x$ and $y$ in the tree, i.e. the unique path between $x$ and $y$ intersects $A$.
\end{proof}

One may object that the proposed algorithm does not exploit the structure of the matrix $Z$, in particular the sparsity of $Z^{-1}$. It is therefore natural to expect that further improvements could be achieved.

In addition, the question arises as to whether the algorithm remains correct when applied to other classes of finite metric spaces $X$ beyond weighted trees. To avoid pathological cases, one must at least assume that all submatrices $Z_A$ are invertible for $A \subseteq X$ (this is required for solving the linear system at line~6 of the algorithm). This condition is satisfied, for instance, when $X$ is positive definite.

In any case, correctness of the output can always be verified in polynomial time using the criterion of Proposition~\ref{prop:diversity-max-meas}. Numerical experiments suggest that the algorithm remains correct for finite subsets of Euclidean spaces (at least for subsets of $\R^2$ consisting of up to a thousand points).

\begin{openproblem}
  Does \textsc{DiversityMaximizingMeasure} correctly compute the maximizing-diversity measure of all finite subsets of $\R^n$? Of all finite metric spaces of negative type?
\end{openproblem}

\subsection{Consequences for branch points}

The correctness of the algorithm \textsc{DiversityMaximizingMeasure} has the following consequence on branch points.
\begin{cor}
\label{cor:47}
Let $X$ be a weighted tree and let $x \in X$.
\begin{itemize}
\item[\upshape{(a)}] The point $x$ does not belong to the support of the maximising measure whenever
\begin{equation}
\label{eq:excluX}
\sum_{e \ni x} \frac{1}{1 + e^{-\ell(e)}} \leq \deg(x) - 1.
\end{equation}
\item[\upshape{(b)}] In the special case where $\ell(e) \leq \log 2$ for all $e \in E$, no branching point of $X$ belongs to the support of the maximising measure.
\end{itemize}
\end{cor}

\begin{proof}
For (a), the condition~\eqref{eq:excluX} is equivalent to $w(x) \leq 0$, where $w$ is the solution found in Proposition~\ref{prop:magnitude_weighted_tree} of the linear system $Zw = \ind_X$. Hence the vertex $x$ is eliminated at the first iteration of Algorithm~\ref{algomax} and therefore does not belong to the support of the diversity-maximizing measure.

For (b), when $\ell(e) \leq \log 2$ for all $e \in E$,
\[
\sum_{e \ni x} \frac{1}{1 + e^{-\ell(e)}} \leq \frac{2}{3} \deg x
\]
and this is less than or equal to $\deg x - 1$ whenever $\deg x \geq 3$.
\end{proof}

The inequality~\eqref{eq:excluX} is never satisfied when $\deg(x)=2$. Nevertheless, vertices of degree~2 may still be excluded from the support of the diversity-maximizing measure, although this does not occur at the first iteration of \textsc{DiversityMaximizingMeasure}.

We now give an alternative proof of Corollary~\ref{cor:47}, based on the following proposition, which provides a sufficient condition for a point not to belong to the support of the diversity-maximizing measure.

\begin{prop}
\label{prop:49}
Let $X$ be a finite positive definite metric space and let $x \in X$. Suppose there exists a measure $\nu$ on $X$ supported in $X \setminus \{x\}$ such that
\[
Z\nu(y) \leq e^{-d(x,y)} \qquad \text{for all } y \in X
\]
Then $x$ does not belong to the support of the diversity-maximizing measure.
\end{prop}

\begin{proof}
Let $\mu$ denote the diversity-maximizing measure, and suppose by contradiction that $x \in \rmsupp \mu$. For $t \geq 0$ sufficiently small, the measure $\mu + t(\nu - \delta_x)$ remains a probability measure.

We compute the first variation
\begin{equation}
\label{eq:first_variation}
\| \mu + t(\nu - \delta_x) \|_{\calW}^2
= \| \mu \|_{\calW}^2
+ 2t \int_X Z(\nu - \delta_x)(y)\,\mu(dy)
+ o(t)
\end{equation}
Now observe that $Z(\nu - \delta_x)(y) = Z\nu(y) - e^{-d(x,y)} \leq 0$ for all $y \in X$. Moreover, at the point $x \in \rmsupp \mu$, we have the strict inequality
\[
Z(\nu - \delta_x)(x)
= Z\nu(x) - 1
= \int_X \left(e^{-d(x,y)} - 1\right)\,\nu(dy)
< 0
\]
Since $x \in \rmsupp \mu$, this implies that the integral term in~\eqref{eq:first_variation}
is negative. Hence the squared norm strictly increases for small positive $t$, contradicting the optimality of $\mu$.
\end{proof}

Proposition~\ref{prop:49} is not specific to trees. From a computational point of view, finding a measure $\nu$ satisfying the stated criterion amounts to solving a linear programming problem, which has exponential worst-case complexity, although it remains tractable in many practical instances.

Numerical experiments (in the case of weighted trees) indicate, however, that the converse of this proposition does not hold in general.

\begin{proof}[Another proof of Corollary~\ref{cor:47}]
We only prove (a), since (b) is a direct consequence of it.

Let $x$ satisfy condition~\eqref{eq:excluX}, and let $y_1, \dots, y_{\deg x}$ denote the neighbours of $x$ in $X$. We set $\ell_j = d(x, y_j)$ for all $j \in \{1, \dots, \deg x\}$. We construct a probability measure $\nu$, supported on the neighbours of $x$, which will satisfy the assumption of Proposition~\ref{prop:49}.

At $x$, there is nothing to check since $Z\nu(x) \leq 1$. Let now $y \in X \setminus \{x\}$. Then $y$ belongs to a connected component of $X \setminus \{x\}$, say the one containing $y_i$. We obtain
\[
Z\nu(y)
= \sum_{j=1}^{\deg x} \nu(\{y_j\}) e^{-d(y, y_j)}
= e^{-d(x, y)} \left(
\nu(\{y_i\}) e^{\ell_i}
+ \sum_{\substack{j=1 \\ j \neq i}}^{\deg x} \nu(\{y_j\}) e^{-\ell_j}
\right)
\]
since $d(y, y_i) = d(x, y) - \ell_i$ and $d(y, y_j) = d(x, y) + \ell_j$ for $j \neq i$.

We introduce the constant $S = \sum_{j=1}^{\deg x} \nu(\{y_j\}) e^{-\ell_j}$
so that
\[
Z\nu(y)
= e^{-d(x, y)} \left(
\nu(\{y_i\}) (e^{\ell_i} - e^{-\ell_i}) + S
\right)
\]

This suggests choosing, for all $j \in \{1, \dots, \deg x\}$,
\[
\nu(\{y_j\})
= \frac{C}{e^{\ell_j} - e^{-\ell_j}}
= \frac{C e^{-\ell_j}}{1 - e^{-2\ell_j}}
\]
where $C$ is chosen so that $\nu(X) = 1$.

With this choice, the condition of Proposition~\ref{prop:49} becomes $C + S \leq 1$, equivalently $1 + S/C = 1/C$, that is,
\[
1 + \sum_{j=1}^{\deg x} \frac{e^{-2\ell_j}}{1 - e^{-2\ell_j}}
\leq
\sum_{j=1}^{\deg x} \frac{e^{-\ell_j}}{1 - e^{-2\ell_j}}
\]
which is equivalent to~\eqref{eq:excluX} by elementary computations.
\end{proof}

\section{Support of the diversity-maximizing measure on compact $\R$-trees}
\label{sec:support}

Our first result is of a general nature: branch points in a compact $\R$-tree do not contribute to diversity. The proof of the following theorem is based on the idea of the second proof of Corollary~\ref{cor:47}, but it is more technical.
\begin{thm}
\label{thm:branch}
The branch points of a compact $\R$-tree do not belong to the support of its diversity-maximizing measure.
\end{thm}

\begin{proof}
Let $X$ be the compact $\R$-tree under consideration and fix a branching point $x_0 \in X$, which we assume by contradiction to belong to the support of the diversity-maximising measure $\mu$. It is therefore possible to find $L > 0$ and three points $y_1, y_2, y_3$ at distance $L$ from $x_0$, belonging to three distinct connected components of $X \setminus \{x_0\}$. Note that it is always possible to decrease the value of $L$, which we will do whenever new conditions on $L$ arise.

Let $\varepsilon \in (0,1)$, to be chosen later. We define the probability measures
\[
\tilde{\mu} \colon A \mapsto \frac{\mu(B(x_0, \varepsilon L) \cap A)}{\mu(B(x_0, \varepsilon L))}
\qquad \text{ and } \qquad
\nu = \frac{\delta_{y_1} + \delta_{y_2} + \delta_{y_3}}{3}
\]
Our goal is to show that $\langle \mu, \nu - \tilde{\mu} \rangle_{\calW} < 0$, which will contradict the optimality of $\mu$, since for $t \ge 0$ sufficiently small the measure $\mu + t(\nu - \tilde{\mu})$ is a probability measure and
\[
\| \mu + t(\nu - \tilde{\mu}) \|_{\calW}^2
= \| \mu \|_{\calW}^2
+ 2t \langle \mu, \nu - \tilde{\mu} \rangle_{\calW}
+ o(t).
\]

Let $x \in B(x_0, L)$. We have
\begin{align*}
Z(\nu - \tilde{\mu})(x)
&= \frac{1}{3} \sum_{i=1}^3 \left( e^{-d(x,y_i)} - Z\tilde{\mu}(x) \right) \\
&= \frac{1}{3} \sum_{i=1}^3 \left( e^{-d(x,y_i)} - \int e^{-d(x,y)}\, \tilde{\mu}(dy) \right) \\
&= \frac{1}{3} \sum_{i=1}^3 \left( e^{-d(x,y_i)} - e^{-d(x_0,x)} + \int \left( e^{-d(x_0,x)} - e^{-d(x,y)} \right) \, \tilde{\mu}(dy) \right) \\
&\le \frac{1}{3} \sum_{i=1}^3 \left( e^{-d(x,y_i)} - e^{-d(x_0,x)} \right) + \varepsilon L
\end{align*}
since the function $y \mapsto e^{-d(x,y)}$ is $1$-Lipschitz and $d(x_0,y) \le \varepsilon L$ for all $y \in \rmsupp \tilde{\mu}$.

We first check that if $x \in \lseg x_0, y_i \rseg$ for some $i \in \{1,2,3\}$, and writing $\ell = d(x_0,x)$, then
\[
\frac{1}{3} \sum_{i=1}^3 \left(e^{-d(x,y_i)} - e^{-d(x_0,x)}\right)
= \frac{e^{-(L-\ell)} - e^{-\ell} + 2\left(e^{-L-\ell} - e^{-\ell}\right)}{3}.
\]
A straightforward study of this function shows that it is non-positive for all $\ell \in [0,L]$ whenever $L \leq \log 2$.

When $x$ does not belong to the tripod $\bigcup_{i=1}^3 \lseg x_0, y_i \rseg$, the situation is even simpler: all three terms in the sum are non-positive, since $d(x,y_i) = d(x,x_0) + d(x_0,y_i)$ for all $i$.

In all cases, we obtain
\begin{equation}
\label{eq:znummuB}
\forall x \in B(x_0,L), \qquad Z(\nu - \tilde{\mu})(x) \leq \varepsilon L
\end{equation}

Assume now that $x \notin B(x_0,L)$. We first suppose that $x$ lies in the connected component of $y_1$, $y_2$ or $y_3$ in $X \setminus \{x_0\}$, say that of $y_1$ without loss of generality. In this case, $x \notin \lseg x_0,y_1\rseg$, and therefore
\[
d(x,y_i) =
\begin{cases}
d(x_0,x) - L & \text{if } i = 1 \\
d(x_0,x) + L & \text{if } i \in \{2,3\}
\end{cases}
\]
We deduce that
\[
Z\nu(x) = e^{-d(x_0,x)} \left( \frac{e^L + 2e^{-L}}{3} \right)
\]
The case where $x$ does not belong to any of the connected components of $y_1,y_2,y_3$ in $X \setminus \{x_0\}$ is simpler: in this situation, $Z\nu(x) = e^{-d(x,x_0)} e^{-L}$.
Since one easily checks that $e^{-L} \leq (e^L + 2e^{-L})/3$,
we obtain the estimate
\[
\forall x \notin B(x_0,L), \qquad
Z\nu(x) \le e^{-d(x_0,x)} \left( \frac{e^L + 2e^{-L}}{3} \right)
\]
Moreover, for $x \notin B(x_0,\varepsilon L)$, we have
\[
Z\tilde{\mu}(x) = \int_X e^{-d(x,y)}\,\tilde{\mu}(dy)
\ge e^{-\varepsilon L} e^{-d(x_0,x)}
\]
since $d(x,y) \le d(x_0,x) + \varepsilon L$ for all $y \in \rmsupp \tilde{\mu}$. We thus obtain
\[
\forall x \notin B(x_0,\varepsilon L), \qquad
Z(\nu - \tilde{\mu})(x)
\le e^{-d(x_0,x)} \left( \frac{e^L + 2e^{-2L}}{3} - e^{-\varepsilon L} \right)
\]
Using~\eqref{eq:znummuB} together with the previous estimate, we obtain
\begin{align*}
\langle \mu, \nu - \tilde{\mu} \rangle_{\calW}
&= \int_X Z(\nu - \tilde{\mu})(x)\,\mu(dx) \\
&= \int_{B(x_0,L)} Z(\nu - \tilde{\mu})(x)\,\mu(dx)
+ \int_{X \setminus B(x_0,L)} Z(\nu - \tilde{\mu})(x)\,\mu(dx) \\
&\le \varepsilon L\, \mu(B(x_0,L))
+ \int_{X \setminus B(x_0,L)} e^{-d(x_0,x)}\,\mu(dx)
\left( \frac{e^L + 2e^{-2L}}{3} - e^{-\varepsilon L} \right)
\end{align*}
All our constructions (notably $\nu$ and $\tilde{\mu}$) depend on the choice of $L$ and $\varepsilon$, even if this is not reflected in the notation. Letting $L \to 0$ with $\varepsilon$ fixed, the above inequality yields
\begin{align*}
\limsup_{L \to 0} \frac{\langle \mu, \nu - \tilde{\mu} \rangle_{\calW}}{L}
&\le \varepsilon\, \mu(\{x_0\})
- \int_{X \setminus \{x_0\}} e^{-d(x_0,x)}\,\mu(dx)\left( 1 - \varepsilon \right)
\end{align*}

The integral appearing on the right-hand side cannot vanish. Otherwise, we would necessarily have $\mu = \delta_{x_0}$, and the maximum diversity would equal $1$, which is impossible (except in the trivial case where $X$ consists of a single point, in which case no branch point exists). By choosing $\varepsilon > 0$ sufficiently small, in a way that depends only on $\mu$ and that could be fixed in advance, the right-hand side becomes negative. Hence there exists $L > 0$ sufficiently small such that $\langle \mu, \nu - \tilde{\mu} \rangle_{\calW} < 0$, which concludes the proof.
\end{proof}

If $\operatorname{Branch}(X)$ is dense in $X$, Theorem~\ref{thm:branch} implies that $\rmsupp \mu$ is meager and contains either leaves or regular points.

More generally, diversity-maximising measures tend to concentrate near the leaves. They may, of course, assign mass to regular points. This is already visible in the simple case of a segment $[0,L]$, which has no branch points: the measure $(\delta_0 + \delta_L + \calL^1)/(2 + L)$ places two atoms at the leaves, while being more diffuse on regular points.

For compact $\R$-trees with a rich branching structure, the situation can be significantly more involved. One complication is that the set of leaves may be not closed, it may be even dense as is the case for the Brownian random tree. We expect the following to have a positive answer.

\begin{openproblem}
Let $X$ be a compact $\R$-tree such that $\operatorname{Branch}(X)$ is dense (or equivalently, $\operatorname{Leaves}(X)$ is dense). Is it true that the diversity-maximizing measure $\mu$ is concentrated on leaves, that is, in view of Proposition~\ref{prop:facts_Rtree}(e), $\mu(\operatorname{Sk}(X)) = \mu(X \setminus \operatorname{Leaves}(X)) = 0$?
\end{openproblem}

\begin{appendix}
\section{Notation index}
\needspace{5\baselineskip}
\label{sec:notations}
\begin{center}
\begin{longtable}{p{.155\textwidth}p{.63\textwidth}p{.13\textwidth}}
\toprule
Object & Meaning & Ref. \\
\midrule
\endhead
\bottomrule
\endfoot
$B(x, r)$ & closed ball of center $x$ and radius $r$ & \\
$\ind_A$ & Indicator function of a set $A$ & \\
$\delta_x$ & Dirac measure at a point $x$ & \\
$\calL^1$ & Lebesgue measure on $\R$ & \\
$\calH^1$ & One-dimensional Hausdorff measure on a metric space & \\
$\lambda$ & Length measure on an $\R$-tree & Subsec.~\ref{subsec:Rtree} \\
$\varphi_\# \mu$ & Pushforward of a measure $\mu$ by a measurable map $\varphi$ & \\
$C(X)$ & Space of continuous functions on a compact metric space $X$, with supremum norm & \\
$\calM(X)$ & Space of signed Borel measures on a compact metric space $X$ & \\
$\calM_{\mathrm{atom}}(X)$ & Subspace of $\calM(X)$ linearly spanned by the Dirac measures & Subsec.~\ref{subsec:negtype}\\
$\calP(X)$ & Space of probability measures on a compact metric space $X$, equipped with the weak topology & Subsec.~\ref{subsec:max_div} \\
$Z$ & Similarity kernel $Z(x, y) = e^{-d(x,y)}$ on compact metric space $(X, d)$ & Subsec.~\ref{subsec:max_div}\\
$\langle \cdot, \cdot \rangle_\calW$ & Bilinear form associated to the kernel $Z$. & Eq.~\eqref{eq:def_similarity_form} \\
$ \| \cdot \|_\calW$ & Norm on $\calM(X)$ induced by the inner product $\langle \cdot, \cdot\rangle_\calW$ when $X$ is compact and of negative type & \\
$|X|_+$ & Maximum diversity of a compact metric space $X$ & Eq.~\eqref{eq:def_diversity}\\
$|X|$ & Magnitude of a compact metric space $X$ of negative type & Eq.~\eqref{eq:def_magnitude}\\
$\lseg x, y \rseg$ & Geodesic segment between $x$ and $y$ & Subsec.~\ref{subsec:simplicial} \\
$\operatorname{Sk}(X)$ & Skeleton of an $\R$-tree $X$ & Subsec.~\ref{subsec:Rtree} \\
$\operatorname{Branch}(X)$ & Set of branch points of an $\R$-tree $X$ & Subsec.~\ref{subsec:Rtree} \\
$\operatorname{Leaves}(X)$ & Set of leaves of an $\R$-tree $X$ & Subsec.~\ref{subsec:Rtree} \\
$\deg x$ & Degree of a vertex $x$ in a weighted tree, or of a point $x$ in an $\R$-tree & Subsec.~\ref{subsec:wtree},~\ref{subsec:Rtree}
\end{longtable}
\end{center}

\end{appendix}

\phantomsection
\addcontentsline{toc}{section}{References}
{\small
\bibliographystyle{alpha}
\bibliography{biblio.bib}

\end{document}